\documentclass[12pt,reqno]{amsart}

\usepackage{a4wide}
\usepackage{amssymb}
\usepackage{needspace}
\usepackage{hyperref}
\usepackage[inline]{enumitem}
\usepackage{booktabs}
\usepackage{comment}

\usepackage{tikz}
\usetikzlibrary{decorations.pathreplacing, tikzmark}
\usetikzlibrary{matrix,fit,positioning,calc}
\usetikzlibrary{positioning}
\usetikzlibrary{arrows.meta}
\usepackage{amsmath}

\usepackage{pgfplots} \pgfplotsset{compat=1.18}

\newtheorem{theorem}[subsection]{Theorem}

\newtheorem{proposition}[subsection]{Proposition}
\newtheorem*{proposition*}{Proposition}
\newtheorem{example}{Example}
\newtheorem{remark}{Remark}
\newtheorem*{claim*}{Claim}
\newtheorem*{claim}{Claim}

\providecommand{\Z}{\mathbb{Z}}
\providecommand{\N}{\mathbb{N}}
\providecommand{\C}{\mathbb{C}}

\providecommand{\T}{\mathbb{T}}
\providecommand{\E}{\mathop{\mathbb{E}}}
\renewcommand{\P}{\mathop{\mathbb{P}}}
\providecommand{\wh}{\widehat}

\newcommand{\dd}{\,\mathrm{d}}

\DeclareMathOperator{\sinc}{sinc}

\newcounter{constcntbig} 
\newcounter{constcntlittle} 

\makeatletter
\newcommand{\newconstbig}[1]{%
    \refstepcounter{constcntbig}%
    \hypertarget{const:#1}{\mbox{}}%
    \protected@write\@auxout{}%
        {\string\newlabel{const:#1}{{\arabic{constcntbig}}{\thepage}}}%
    C_{\arabic{constcntbig}}%
}
\makeatother
\makeatletter
\newcommand{\newconstlittle}[1]{%
    \refstepcounter{constcntlittle}%
    \hypertarget{const:#1}{\mbox{}}%
    \protected@write\@auxout{}%
        {\string\newlabel{const:#1}{{\arabic{constcntlittle}}{\thepage}}}%
    c_{\arabic{constcntlittle}}%
}
\makeatother

\newcommand{\refconstlittle}[1]{%
    \hyperlink{const:#1}{c_{\ref{const:#1}}}%
}
\newcommand{\refconstbig}[1]{%
    \hyperlink{const:#1}{C_{\ref{const:#1}}}%
}

\numberwithin{equation}{section}

\begin{document}

\title{The Wiener norm of lifts of dense Sidon sets}

\author{Tom Sanders}
\address{Mathematical Institute\\
University of Oxford\\
Radcliffe Observatory Quarter\\
Woodstock Road\\
Oxford OX2 6GG\\
United Kingdom}
\email{tom.sanders@maths.ox.ac.uk}
\begin{abstract}
We describe, with proof, an infinite class of examples of finite sets $A$ for which $\|\wh{1_A}\|_1\geq (c_\mathcal{R}-o_{|A| \rightarrow \infty}(1))\sqrt{|A|}$ where $c_\mathcal{R}=0.953\dots$.
\end{abstract}

\maketitle

\section{Introduction}

We begin with a little notation following \cite{rud::1}. Any discrete Abelian group $G$ has a compact Abelian dual group $\Gamma$.  We write $\|\cdot \|_\infty$ for the supremum norm on $C(\Gamma)$, the space of continuous functions $\Gamma \rightarrow \C$.

The group $\Gamma$ supports a Haar probability integral \cite[\S1.1.1, p.~1]{rud::1} which we denote $\int$. For $1 \leq p <\infty$ we define the $L_p$-norm of $f \in C(\Gamma)$ by
\begin{equation*}
\|f\|_p:=\left(\int{|f(\gamma)|^p\dd \gamma}\right)^{1/p}.
\end{equation*}
We write $\langle x,\gamma\rangle$ for the usual pairing of $x \in G$ and $\gamma \in \Gamma$  \cite[\S1.2.1, (3), p.~7]{rud::1}, and for $f \in \ell_1(G)$ the Fourier transform $\wh{f}$ is in $C(\Gamma)$ and is defined pointwise by
\begin{equation}\label{eqn.ftdef}
\wh{f}(\gamma):=\sum_{x \in G}{f(x)\overline{\langle x,\gamma\rangle}} \text{ for all } \gamma \in \Gamma.
\end{equation}
The map $f \mapsto \|\wh{f}\|_1$ is a norm called the Wiener norm.

We are interested in bounds on $\|\wh{1_A}\|_1$ that depend only on $n=|A|$. When $G=\Z$, $\Gamma=\T$, and $\langle z,\theta\rangle:=\exp(2\pi i z\theta)$ it was conjectured by Littlewood that 
\begin{equation*}
\|\wh{1_A}\|_1 = \Omega(\log n) \text{ for any set }A\text{ of }n\text{ integers.}
\end{equation*}
This was famously proved by Konyagin \cite{kon::} and McGehee, Pigno, and Smith \cite{mcgpigsmi::}.

In the other direction, Bourgain (after \cite[(3.1), p.~61]{bou::7}) asked whether there is an absolute $c_{\textsc{b}}>0$ such that
\begin{equation}
\label{eqn.bb}\|\wh{1_A}\|_1 \leq (1-c_{\textsc{b}}+o_{n \rightarrow \infty}(1))\sqrt{n} \text{ for any set }A\text{ of }n\text{ integers}.
\end{equation}
He was able to use the work of McGehee, Pigno, and Smith to show that 
\begin{equation}\label{eqn.bbound}
\|\wh{1_A}\|_1 \leq \left(1-\Omega\left(\frac{\log n}{n}\right)\right)\sqrt{n}.
\end{equation}
Bourgain's question is the motivation of this paper and, while we do not answer it, our main contribution is a new construction giving a small (but positive) upper bound on $c_\textsc{b}$. This construction comes from lifting examples in other discrete Abelian groups where the analogue of $c_\textsc{b}$ is equal to $0$, and we begin the paper with these in the next section.

\subsection*{AI use} Everything was discussed extensively with AI. The proofs were AI generated and then thoroughly reworked in collaboration with AI. 

\section{The upper bound in discrete Abelian groups}\label{sec.2} For $G$ a discrete Abelian group with dual $\Gamma$ as above, Parseval's theorem \cite[\S1.6.1, Theorem, p.~26]{rud::1} says that
\begin{equation*}
\int{\wh{f}(\gamma)\overline{\wh{g}(\gamma)}\dd \gamma}=\sum_{x \in G}{f(x)\overline{g(x)}} \text{ whenever }f,g \in \ell_1(G).
\end{equation*}
We have restricted to $\ell_1(G)$ so we do not have to discuss extending the Fourier transform from (\ref{eqn.ftdef}) to functions in $\ell_2(G) \setminus \ell_1(G)$.

This immediately gives  $\|\wh{1_A}\|_1 \leq \sqrt{n}$ for any set $A$ of size $n$, but in fact it is possible to give a slight improvement of this. To show this, we require a little more notation: for $f,g \in \ell_1(G)$, the convolution $f \ast g$ is also in $\ell_1(G)$ and is defined pointwise by
\begin{equation*}
 f \ast g(x)=\sum_{y+z=x}{f(y)g(z)}\text{ for all }x \in G.
\end{equation*}
With this notation $\wh{f \ast g}(\gamma)=\wh{f}(\gamma)\wh{g}(\gamma)$ for all $\gamma \in \Gamma$. 

\begin{proposition}\label{prop.upgen}
Suppose that $G$ is a discrete Abelian group and $A \subset G$ has size $n$. Then
\begin{equation*}
\|\wh{1_A}\|_1 \leq \sqrt{n-1}+\frac{1}{n + \sqrt{n-1}}.
\end{equation*}
\end{proposition}
\begin{proof}
This improvement rests on the fact, used in \cite[(3.3), p.~61]{bou::7}, that
\begin{align}\nonumber 
\|\wh{1_A}\|_4^4 =\sum_{z \in G}{1_A \ast 1_{-A}(z)^2}& \geq 1_A \ast 1_{-A}(0)^2 + \sum_{z \neq 0}{1_A \ast 1_{-A}(z)}\\ \label{eqn.l4}  & = n^2 + n(n-1) = 2n^2-n.
\end{align}
The conclusion of the proposition is certainly true if $n=1$ since both sides are $1$ in that case. Assume $n>1$ and write $r:=\sqrt{n-1}$ and
\begin{equation*}
\alpha_0:=\frac{nr(n+2r)}{2(n+r)^2}, \alpha_1:=\frac{n^2+2nr+3r^2}{2r(n+r)^2}, \alpha_2:= \frac{1}{2r(n+r)^2}.
\end{equation*}
Then
\begin{equation*}
\alpha_0+\alpha_1t^2-\alpha_2t^4-t=\frac{(n-t)(t-r)^2(n+2r+t)}{2r(n+r)^2} \geq 0 \text{ for all }t \in [0,n].
\end{equation*}
Since $|\wh{1_A}(\gamma)| \in [0,n]$ for all $\gamma \in\Gamma$ we may take $t=|\wh{1_A}(\gamma)|$ in this inequality and integrate over $\gamma$ to get
\begin{align*}
\|\wh{1_A}\|_1 & \leq \alpha_0 +\alpha_1\|\wh{1_A}\|_2^2 -\alpha_2\|\wh{1_A}\|_4^4\\ & \leq \alpha_0 +\alpha_1n -\alpha_2(2n^2-n)=\sqrt{n-1}+\frac{1}{n + \sqrt{n-1}},
\end{align*}
where we used Parseval's theorem to compute $\|\wh{1_A}\|_2^2$ and (\ref{eqn.l4}) and non-negativity of $\alpha_2$ to bound $-\alpha_2\|\wh{1_A}\|_4^4$. The result is proved.
\end{proof}
\begin{remark}The values of $\alpha_0$, $\alpha_1$, and $\alpha_2$ in the argument above can be found by solving the optimisation problem
\begin{equation*}
\min\{\alpha_0+\alpha_1n-\alpha_2(2n^2-n): \alpha_2 \geq 0, t \leq \alpha_0+\alpha_1t^2-\alpha_2t^4\text{ for all }t \in [0,n]\}.
\end{equation*}
\end{remark}
It turns out that, for infinitely many $n$, there is a discrete Abelian group $G$ and set $A \subset G$ of size $n$ giving equality above. To describe these, recall that a perfect difference set is a set $A \subset G$ of size $n$ such that 
\begin{equation}\label{eqn.ft}
1_A \ast 1_{-A} - (n-1)1_{\{0\}} = 1_G.
\end{equation}
In words, every non-identity element $x \in G$ can be written uniquely as $x=a-a'$ for $a,a' \in A$. In particular, this means that $n(n-1)+1=|G|$.

One might be forgiven for wondering whether there are any perfect difference sets --- it sounds like a rather stringent condition --- but it happens that for $N=q^2+q+1$ with $q$ a prime power, Singer constructed perfect difference sets in $\Z_N:=\Z/N\Z$ in \cite{sin::} (of size $q+1$).
\begin{example}[Perfect difference sets]\label{ex.g}
For $A \subset G$ a perfect difference set, $\|\wh{1_A}\|_1$ attains the bound in Proposition \ref{prop.upgen}. Indeed, taking the Fourier transform of (\ref{eqn.ft}), we get
\begin{equation}\label{eqn.flat}
|\wh{1_A}(\gamma)|=\sqrt{n-1} \text{ for all }\gamma \neq 0.
\end{equation}
Hence
\begin{equation*}
 \|\wh{1_A}\|_1 = \frac{1}{|G|}(n + (|G|-1)\sqrt{n-1}) =\sqrt{n-1}+\frac{1}{n + \sqrt{n-1}}.
\end{equation*}
\end{example}
In particular, these examples show that there is no analogue of Bourgain's result (\ref{eqn.bbound}) for these groups.

\section{Sets of integers with large Wiener norm}  In \cite[(3.1)]{bou::7} Bourgain defines
\begin{equation*}
\beta_n:=\sup{\left\{\frac{1}{\sqrt{n}}\|\wh{1_A}\|_1: A \subset \Z \text{ and }|A|=n\right\}}
\end{equation*}
and then asks if $\Sigma:=\sup_{n>1}{\beta_n}<1$. Since $\beta_n<1$ for all $n>1$ (for example, by Proposition \ref{prop.upgen}) this is equivalent to asking if
\begin{equation*}
c_\textsc{b}:=1-\limsup_{n \rightarrow \infty}{\beta_n}
\end{equation*}
has $c_\textsc{b}>0$. $\Sigma$ can be seen to be quite large as a result of what happens for small values of $n$ \emph{e.g.\ }
\begin{equation*}
\beta_2=\frac{1}{\sqrt{2}}\int_0^1{|1+\exp(2\pi i \theta)|\dd\theta} = \frac{2\sqrt{2}}{\pi}=0.900\dots
\end{equation*}

To find upper bounds on the constant $c_\textsc{b}$ it is natural to start by seeing what happens when $A=\{a_1,\dots,a_n\}$ is a set of $n$ `sufficiently independent' integers. In this case we expect $\frac{1}{\sqrt{n}}\wh{1_A}$ to have a law that is approximately the law of $\frac{1}{\sqrt{n}}(X_1+\cdots + X_n)$, where the $X_j$s are independent Steinhaus random variables, meaning independent uniform $S^1$-valued random variables.

By the Central Limit Theorem we expect this to converge to the law of $Z \sim N_\C(0,1)$ (\emph{i.e.\ }$Z=X+iY$ where $X,Y \sim N(0,1/2)$) as $n \rightarrow \infty$, and hence
\begin{equation}\label{eqn.converge}
\frac{1}{\sqrt{n}}\|\wh{1_A}\|_1 \rightarrow \E{|Z|}=\frac{\sqrt{\pi}}{2} \text{ as }n \rightarrow \infty.
\end{equation}
The meaning of `sufficiently independent' can be made precise, and this is done by Aistleitner in \cite{Aistleitner:2013aa} where it is shown that if $a_1,a_2,\dots$ is a lacunary sequence, meaning there is some $q>1$ such that $a_{k+1}>qa_k$ for all $k$, then $A=\{a_1,\dots,a_n\}$ satisfies (\ref{eqn.converge}). In particular,
\begin{equation}\label{eqn.ais}
c_\textsc{b}\leq 1-\frac{\sqrt{\pi}}{2}=0.113\dots,
\end{equation}
 which is the content \cite[Theorem 1, p.~682]{Aistleitner:2013aa} though that result is stated about $\Sigma$ as defined above.

In view of Example \ref{ex.g} it is natural to ask if there are perfect difference sets for the integers. There are not, since such sets only exist in finite groups, but there is an analogue: a perfect Golomb ruler of length $N$ is a set $A \subset \Z$ such that
\begin{equation}\label{eqn.fi}
1_A \ast 1_{-A} = 1_{\{-N,\dots,N\}} + (|A|-1)1_{\{0\}}.
\end{equation}
Taking the Fourier transform and noting that the Dirichlet kernel -- exactly the Fourier transform of $1_{\{-N,\dots,N\}}$ -- takes negative values whose magnitude is proportional to $N$ we see that such sets cannot exist for arbitrarily large values of $N$ and in fact \cite{Golomb:1972aa} showed that the largest $N$ for which they do exist is $N=6$ (in this case we may take \emph{e.g.\ }$A=\{0,1,4,6\}$).

Singer's construction produces perfect difference sets inside cyclic groups, and we can ask what happens if we lift those sets to the integers. This idea is not new, having appeared ten years ago in the first version of \cite{Houcein-el-Abdalaoui:2023aa}. The aim of \cite[Theorem 2.1]{Houcein-el-Abdalaoui:2023aa} is to show that $\|\wh{1_A}\|_1 \sim \sqrt{n}$ when $A$ is the lift of a Singer perfect difference set of size $n$, and this would imply $c_\textsc{b}=0$. The strategy for proving this in \cite[\S5]{Houcein-el-Abdalaoui:2023aa} makes use of the Marcinkiewicz-Zygmund inequalities for interpolation. We shall use these in \S\ref{sec.limitations} to show that for a wide class of sets that are close to perfect difference sets, their lifts are \emph{not} examples of sets $A$ of size $n$ with $\|\wh{1_A}\|_1 \sim \sqrt{n}$.

Our main result -- Theorem \ref{thm.mi}  -- shows that a slight variant of Singer's construction lifts to examples of sets showing that
\begin{equation}\label{eqn.new}
c_\textsc{b} \leq 1- \int_0^1{\E{\left|\sum_{m \in \Z}{Y_m\sinc (\pi (t-m))}\right|}\dd t}=0.046\dots
\end{equation}
Here $Y_m$ are independent Steinhaus random variables and $\sinc(x):=\frac{\sin x}{x}$ extended by continuity at $x=0$. This improves on (\ref{eqn.ais}).

We do not know of a simpler expression for the right hand term. Its value can be computed to any desired precision by truncating the sum, which leaves a finite number of random variables and the resulting integral-expectation can be approximated by a Riemann sum. Alternatively it can be estimated by sampling.

\section{The Wiener norm of lifts of certain difference sets}\label{sec.lift}
We will be considering sets $A \subset [N]:=\{1,\dots,N\}$, that arise as lifts of subsets of $\Z_N$. The crucial identity for us is
\begin{align}
\nonumber f_A(\theta)&:=\frac{1}{\sqrt{p}}\wh{1_A}(\theta)-\frac{n}{\sqrt{p}N}\wh{1_{[N]}}(\theta)\\ &=\frac{1}{N}\sum_{j=1}^{N-1}{\frac{1}{\sqrt{p}}\wh{1_A}\left(\frac{j}{N}\right)\wh{1_{[N]}}\left(\theta-\frac{j}{N}\right)} \text{ for all }\theta \in \T.\label{eqn.flarge}
\end{align}
The second equality here can be checked by inserting the definition of the Fourier transform and interchanging the order of summation.

The set $A$ will be chosen so that $\|f_A\|_1$ is well approximated by the expected $L_1$-norm of 
\begin{equation}\label{eqn.Fmod}
F(\theta):=\frac{1}{N}\sum_{j=1}^{N-1}{X_j\wh{1_{[N]}}\left(\theta-\frac{j}{N}\right)} \text{ for all }\theta \in \T,
\end{equation}
where the $X_j$s are independent Steinhaus random variables. This expectation can in turn be simplified:
\begin{proposition}\label{prop.stein}
For $F$ as in (\ref{eqn.Fmod}) we have
\begin{equation*}
\int_0^1{\E{|F(\theta)|}\dd \theta} \rightarrow c_\mathcal{R}:=\int_0^1{\E{\left|\sum_{m\in \Z}{Y_m\sinc (\pi(s-m))}\right|}\dd s} \text{ as }N \rightarrow \infty
\end{equation*}
where $(Y_m)_{m \in \Z}$ are independent Steinhaus random variables.
\end{proposition}
We shall also need a rougher estimate for the $L_{2k}$-norms of $F$:
\begin{proposition}\label{prop.stein2k}
For $F$ as in (\ref{eqn.Fmod}) and $k\in \N$ we have
\begin{equation*}
\int_0^1\E|F(\theta)|^{2k}\dd\theta\leq k!.
\end{equation*}
\end{proposition}
These two results are proved in \S\ref{sec.stein}.

\subsection{Definition of $A$} The Fourier coefficients $\wh{1_A}(j/N)$ are the Fourier coefficients of the image of $A$ under the quotient map $\Z \rightarrow \Z_N$. We shall take $A$ to be a set of a type introduced by Ruzsa in \cite[Theorem 4.4, p267]{ruz::8} as examples sets that are close to being perfect difference sets.

This construction makes the Fourier coefficients $\wh{1_A}(j/N)$ behave like Gauss sums. Gauss sums have good equidistribution properties -- see \cite[Chapter 9]{Katz:1988aa} -- which is why we expect (\ref{eqn.Fmod}) to provide a good approximation to $f_A(\theta)$. In practice we shall not need the full strength of equidistribution and instead make use of classical Weil bounds.

Turning to the definition, suppose that $p$ is prime and write $N:=p(p-1)$, and suppose that $s$ is a primitive root modulo $p$. Then set
\begin{equation}\label{eqn.ruzsa}
A:=\{x \in [N] : x \equiv pt - (p-1)s^t\text{ for some } t \in [p-1]\},
\end{equation}
which has size $n:=p-1$. In words, $A$ is the lift of the image of the graph of $\Z_{p-1} \rightarrow \Z_p; t \mapsto s^t$ under the isomorphism $\Z_{p-1} \times \Z_p \rightarrow \Z_N; (a,b) \mapsto pa-(p-1)b$.

\subsection{Gauss sums} For $\psi \in \wh{\Z_p}$ and $\chi \in \wh{\Z_p^\times}$ define the Gauss sum
\begin{equation*}
G(\psi,\chi):=\sum_{x \in \Z_p^\times}{\psi(x)\chi(x)}.
\end{equation*}
Let $\chi \in \wh{\Z_p^\times}$ be such that $\chi(s)=\exp(-2\pi i/(p-1))$; and let $\psi_b(x):=\exp(2\pi i bx/p)$. Then
\begin{equation}\label{eqn.above}
\wh{1_A}\left(\frac{pa-(p-1)b}{N}\right)=G(\psi_b,\chi^a) \text{ for all }a,b\in \Z.
\end{equation}
Conventionally we extend multiplicative characters $\gamma$ on $\Z_p^\times$ to $\Z_p$ by setting $\gamma(0)=0$, and this means
\begin{equation*}
G(\psi_b,\chi^a) =\begin{cases} G(\psi_1,\chi^a)\chi^{-a}(b) & \text{ if }p-1 \nmid a \text{ or }p \nmid b\\ p-1 & \text{ if }p-1 \mid a \text{ and }p \mid b\end{cases}.
\end{equation*}
The usual Gauss sum calculation tells us that if $p-1 \nmid a$ then
\begin{equation}\label{eqn.gaussnormal}
\left|G(\psi_1,\chi^a)\right|=\sqrt{p},
\end{equation}
and of course $G(\psi_1,\chi^0)=-1$.

With these definitions we have:
\begin{proposition}\label{prop.moments}
Suppose that $A$ is defined as in (\ref{eqn.ruzsa}) and $F$ as in (\ref{eqn.Fmod}) with $N=p(p-1)$. Then
\begin{equation*}
\int_0^1{|f_A(\theta)|\dd\theta} = \int_0^1{\E{|F(\theta)|}\dd\theta} + o_{p \rightarrow \infty}(1).
\end{equation*}
\end{proposition}
This is the main work and is proved in \S\ref{sec.proofmoments}. With these results we then have the main theorem:
\begin{theorem}\label{thm.mi}
Suppose that $A$ is defined as in (\ref{eqn.ruzsa}). Then
\begin{equation*}
\|\wh{1_A}\|_1 \geq (c_\mathcal{R} - o_{n \rightarrow\infty}(1))\sqrt{n} \text{ where }c_\mathcal{R}=\int_0^1{\E{\left|\sum_{m\in \Z}{Y_m\sinc (\pi(s-m))}\right|}\dd s} =0.953\dots.
\end{equation*}
\end{theorem}
\begin{proof}
Recall $N=p(p-1)$ and $n=p-1$. From Propositions \ref{prop.moments} \& \ref{prop.stein} we have $\|f_A\|_1\geq c_\mathcal{R}-o_{p \rightarrow \infty}(1)$. Since $\|\wh{1_{[N]}}\|_1=O(\log N)$ \cite[(12.1), Chapter II, Vol.~1, p.~67]{zyg::}, we then get the result by the triangle inequality applied to the definition of $f_A$ in (\ref{eqn.flarge}).
\end{proof}

\subsection{The proof of Proposition \ref{prop.moments}}\label{sec.proofmoments}
The aforementioned bound of Weil \cite{wei::1} that we shall use is:
\begin{theorem}[Weil bound] \label{thm.weilbound}
Suppose that $p$ is prime, $\chi \in \wh{\Z_p^\times}$ and $f(t)=\prod_{i=1}^d{(t-\xi_i)^{a_i}}$ where $\xi_1, \dots, \xi_d \in \Z_p$ are distinct, $a_1,\dots,a_d \in \Z$, and at least one $a_i$ is not a multiple of the order of $\chi$. Then
\begin{equation*} \left|\sum_{t \in \Z_p}{\chi(f(t))}\right| \leq (d-1)\sqrt{p}
\end{equation*} where we extend $\chi$ by taking $\chi(0)=\chi(\infty)=0$.
\end{theorem}
To prove Proposition \ref{prop.moments} we shall first show that the relevant $L_{2k}$-norms are equal (in the limit as $p \rightarrow \infty$) -- this is Proposition \ref {prop.moments2k} below -- and then use the fact that these are also sufficiently small (Proposition \ref{prop.stein2k}) so that their limit arises from a unique distribution, which means the $L_1$-norms are also the same in the limit. We shall see the details of this later; the heart of the whole argument is:
\begin{proposition}\label{prop.moments2k}
Suppose that $A$ is defined as in (\ref{eqn.ruzsa}), $F$ as in (\ref{eqn.Fmod}), $N=p(p-1)$, and $k\in \N$. Then 
\begin{equation}\label{eqn.mainest2}
\int_0^1{|f_A(\theta)|^{2k}\dd\theta} = \int_0^1{\E|F(\theta)|^{2k}\dd\theta}+o_{k;p\to\infty}(1).
\end{equation}
\end{proposition}
\begin{proof}
Write
\begin{equation}\label{eqn.Hjj}
H(j,j') :=\int_0^1{\prod_{i=1}^k\wh{1_{[N]}}\left(\theta-\frac{j_i}{N}\right) \overline{\prod_{i=1}^{k}\wh{1_{[N]}}\left(\theta-\frac{j_i'}{N}\right)} \dd\theta} \text{ for }j,j' \in\Z^{k}.
\end{equation}
The definition is invariant under translation of any $j_i$ (or $j_i'$) by a multiple of $N$, and so we also think of $H$ as a function on $\Z_N^k\times \Z_N^k$. 

From the definition of (\ref{eqn.Fmod}) and the fact that $\E{X_i^r\overline{X_i}^s} =1_{[r=s]}$, we have
\begin{equation}\label{eqn.Fmomentexpansion}
\int_0^1{\E{|F(\theta)|^{2k}}\dd\theta} = \frac1{N^{2k}} \sum_{\substack{j,j' \in (\Z_N^*)^k\\ \exists \pi \in S_k \text{ s.t.\ }j=\pi^*(j')}}{H(j,j')}.
\end{equation}
Here, we have written $\pi^*(j')$ to mean the vector whose $i$th component is $j_{\pi(i)}'$ . Another way of saying $j=\pi^*(j')$ is that $\{j_1,\dots,j_k\}=\{j_1',\dots,j_k'\}$ as multisets.

Note that $j$ and $j'$ are taken in $(\Z_N^*)^k$, where $\Z_N^*:=\Z_N\setminus \{0\}$, rather than $\Z_N^k$ because the sum in (\ref{eqn.Fmod}) is over non-zero residues.

To understand $H$ we will use two standard estimates:
\begin{equation}\label{eqn.dirichletbound2}
|\wh{1_{[N]}}(\theta)| = O\left(\min\left(N,\frac1{\|\theta\|_{\T}}\right)\right) \text{ where }\|\theta\|_\T:=\min\{|\theta-z|: z \in \Z\},
\end{equation}
and hence
\begin{equation}\label{eqn.dirichletsum2}
\sum_{j \in \Z_N}{\left|\wh{1_{[N]}}\left(\theta-\frac{j}{N}\right)\right|} =O( N\log N).
\end{equation}
Later, in (\ref{eqn.fA}), we shall see that $\|f_A\|_{L_{2k}}^{2k}$ is a weighted sum of $H(j,j')$. Most of the diagonal terms, meaning those pairs $(j,j')$ with $j=\pi^*(j')$ for some $\pi \in S_k$, get weight $1$ which gives a contribution of the same size as (\ref{eqn.Fmomentexpansion}). The off-diagonal terms get small weight and to show that this makes their total contribution small we need the $\ell_1$-norm of $H$ to be small. This is what the next two claims show in a form that is useful to us.
\begin{claim*}
For $r \in [k]$ we have
\begin{equation}\label{Hjj}
\sum_{\substack{j,j' \in \Z_N^k\\  p \mid j_r}}{|H(j,j')|} ,\sum_{\substack{j,j' \in \Z_N^k\\  p-1 \mid j_r}}{|H(j,j')|} = O_k(p^{-1}N^{2k}\log^{2k} N).
\end{equation}
\end{claim*}
\begin{proof}
This is two separate, but similar estimates. We start by proving the upper bound for the first sum. The proof of the second is the same with the roles of $p$ and $p-1$ swapped. Use (\ref{eqn.dirichletsum2}) in the $2k-1$ variables $j_1,\dots,j_{r-1},j_{r+1},\dots,j_k,j_1',\dots,j_k'$ and write $j_r=px$ to get
\begin{equation*}
\sum_{\substack{j,j' \in \Z_N^k\\  p \mid j_r}}{|H(j,j')|} = O(N\log N)^{2k-1}\sum_{x \in \Z_{p-1}}{\int_0^1{\left|\wh{1_{[N]}}\left(\theta-\frac{x}{p-1}\right)\right|\dd \theta}}.
\end{equation*}
The integral is then the $L_1$-norm of the Dirichlet kernel (after a change of variables $\theta-\frac{x}{p-1} \mapsto \theta$) and so is $O(\log N)$, and $|\Z_{p-1}|=p-1=\frac{N}{p}$. Combining all this gives the claim.
\end{proof}
Similarly, but harder, we have:
\begin{claim*}
For $r,s \in [k]$ we have
\begin{equation}\label{Hjj2}
\sum_{\substack{j,j' \in \Z_N^k\\  j_r \neq j_s' \text{ and }p \mid j_r-j_s'}}{|H(j,j')|} = O_k(p^{-1}N^{2k}\log^{2k} N).
\end{equation}
\end{claim*}
\begin{proof}
Write $S$ for the sum on the left in (\ref{Hjj2}). By (\ref{eqn.dirichletsum2}) in the $2k-2$ variables $j_1,\dots,j_{r-1},j_{r+1},\dots,j_k,j_1',\dots,j_{s-1}',j_{s+1}',\dots j_k'$ we have
\begin{equation*}
S= O(N\log N)^{2k-2}\cdot\sum_{\substack{j_r,j_s' \in \Z_N\\ j_r \neq j_s'\text{ and }p \mid j_r-j_s'}}{\int_0^1{ \left|\wh{1_{[N]}}\left(\theta - \frac{j_r}{N}\right)\right|\left|\wh{1_{[N]}}\left(\theta - \frac{j_s'}{N}\right)\right|\dd \theta}}.
\end{equation*}
By change of variables we have
\begin{equation*}
\int_0^1{ \left|\wh{1_{[N]}}\left(\theta - \frac{j_r}{N}\right)\right|\left|\wh{1_{[N]}}\left(\theta - \frac{j_s'}{N}\right)\right|\dd \theta} = \int_0^1{ \left|\wh{1_{[N]}}\left(\theta\right)\right|\left|\wh{1_{[N]}}\left(\theta - \frac{j_r-j_s'}{N}\right)\right|\dd \theta}.
\end{equation*}
Hence, writing $px=j_r-j_s'$, we have
\begin{align}\label{eqn.Su}
S& =O(N\log N)^{2k-2} \cdot N \cdot \sum_{x\in \Z_{p-1}^*}{\int_0^1{|\wh{1_{[N]}}(\theta)|\left|\wh{1_{[N]}}\left(\theta - \frac{x}{p-1}\right)\right|\dd \theta}}.
\end{align}
Now split the range of integration into
\begin{equation*}
 I_1:=\{\theta \in [0,1]: \|\theta\|_\T\leq \frac1{2(p-1)}\},
 \end{equation*}
 \begin{equation*}
 I_2:=\left\{\theta\in [0,1]: \left\|\theta - \frac x{p-1}\right\|_\T\leq \frac1{2(p-1)}\right\},
 \end{equation*}
 and 
 \begin{equation*}
 I_3:=[0,1]\setminus (I_1 \cup I_2).
 \end{equation*}
 For $x \in \Z_{p-1}^*$ write $d(x):=(p-1)\left\|\frac{x}{p-1}\right\|_\T$. By the triangle inequality for $\|\cdot \|_\T$, if $\theta \in I_1$ then $\left\|\theta - \frac x{p-1}\right\|_\T\geq \frac{d(x)}{2(p-1)}$, and similarly if $\theta \in I_2$ then $\left\|\theta \right\|_\T\geq \frac{d(x)}{2(p-1)}$. It follows from (\ref{eqn.dirichletbound2}) that
\begin{align}
\nonumber \int_{I_1}{|\wh{1_{[N]}}(\theta)|\left|\wh{1_{[N]}}\left(\theta - \frac{x}{p-1}\right)\right|\dd \theta}&  = O\left(\frac{2(p-1)}{d(x)}\right)\cdot \int_{I_1}{\min\left\{N,\frac{1}{\|\theta\|_\T}\right\}\dd \theta}\\ \label{eqn.I1e} &=O\left(\frac{p}{d(x)}\log N\right),
\end{align}
and a similar estimate for the integral over $I_2$. Finally, the triangle inequality tells us that $\left\| \frac x{p-1}\right\|_\T\leq \|\theta\|_\T+\left\|\theta - \frac x{p-1}\right\|_\T$, which rearranges to
\begin{equation*}
\frac1{\|\theta\|_\T\left\|\theta - \frac x{p-1}\right\|_\T} \leq \frac{1}{\left\| \frac x{p-1}\right\|_\T}\left(\frac{1}{\|\theta\|_\T} + \frac{1}{\left\|\theta - \frac x{p-1}\right\|_\T}\right).
\end{equation*}
Hence by (\ref{eqn.dirichletbound2}) we have
\begin{align}\nonumber 
\int_{I_3}{|\wh{1_{[N]}}(\theta)|\left|\wh{1_{[N]}}\left(\theta - \frac{x}{p-1}\right)\right|\dd \theta}& =O\left(\frac{p-1}{d(x)}\right)\left(\int_{I_3}{\frac{1}{\|\theta\|_\T}\dd\theta} +\int_{I_3}{\frac{1}{\left\|\theta-\frac{x}{p-1}\right\|_\T}\dd\theta} \right)\\ \label{eqn.I3e} &= O\left(\frac{p}{d(x)}\log p\right)=O\left(\frac{p}{d(x)}\log N\right).
\end{align}
Inserting the bounds from (\ref{eqn.I1e}), its analogue for $I_2$ in place of $I_1$, and (\ref{eqn.I3e}) into (\ref{eqn.Su}) gives
\begin{align*}
S& =O(N\log N)^{2k-2} \cdot N \cdot \sum_{x\in \Z_{p-1}^*}{O\left(\frac{p}{d(x)}\log N\right)}.
\end{align*}
Summing over $x \in \Z_{p-1}^*$, $d(x)$ takes each of the values $1,2,\dots,\frac{p-1}{2}$ twice, so the sum is $O(p\log N\log p)=O(p^{-1}N\log^2 N)$ which gives the claim.
\end{proof}
Now we turn to the $L_{2k}$-norm of $f_A$. For $j \in \Z_N^k$ write
\begin{equation*}
g(j) :=\begin{cases} \prod_{i=1}^k {\frac{\wh{1_A }\left(\frac{j_i}{N}\right)}{\sqrt p}} & \text{ for }j \in (\Z_N^*)^k\\ 0 & \text{ otherwise}.\end{cases}
\end{equation*}
Given $j \in \Z_N^k$ there are unique $a \in \Z_{p-1}^k$ and $b \in \Z_p^k$ such that $j_i \equiv pa_i - (p-1)b_i \bmod N$ for all $i \in [k]$ (and similarly for $j'$, $a'$, and $b'$). This is the same instance of the Chinese Remainder Theorem implicitly mentioned after the definition (\ref{eqn.ruzsa}). Then
\begin{equation}\label{eqn.expg}
g(j)=p^{-\frac k2}\prod_{i=1}^k{G(\psi_1,\chi^{a_i})\chi^{-a_i}(b_i)} \text{ for all }j \in \Z_N^k.
\end{equation}
To see this, (\ref{eqn.above}) deals with $j \in \Z_N^*$, and for $j \in \Z_N^k \setminus (\Z_N^*)^k$ we use the fact that $\chi^{-0}(0)=0$ by definition. Note that $\chi$ in (\ref{eqn.expg}) is the particular choice of $\chi$ just before (\ref{eqn.ruzsa}), \emph{i.e.\ }$\chi(s)=\exp(-2\pi i/(p-1))$ where $s$ is a primitive root modulo $p$, fixed for our definition of $A$ just before (\ref{eqn.ruzsa}).

The function $g$ arises from (\ref{eqn.flarge}) which gives
\begin{equation}\label{eqn.fA}
\int_0^1|f_A(\theta)|^{2k}\dd\theta
=
\frac1{N^{2k}}
\sum_{j,j' \in \Z_N^k}{g(j)\overline{g(j')}H(j,j')}.
\end{equation}

Recalling the definition of $H$ in (\ref{Hjj}), the change of variables $\theta \mapsto \theta + \frac{h}{N}$ shows that
\begin{equation}\label{eqn.hshift}
H(j+h\textbf 1, j'+h\textbf 1)=H(j,j') \text{ for all }h \in \Z_N, j,j' \in \Z_N^k,
\end{equation}
where $\textbf 1 = (1,\dots,1)$. Let 
\begin{equation*}
\mathcal{S}:=\{(j,j') \in \Z_N^k\times \Z_N^k: \exists \pi \in S_k \text{ s.t.\ }b=\pi^*(b')\}.
\end{equation*}
Note that $\mathcal{S}$ is closed under $(j,j')\mapsto (j+h\textbf 1,j'+h\textbf 1)$, and hence so is $\mathcal{S}^c=\Z_N^k \times \Z_N^k \setminus \mathcal{S}$.
\begin{claim*}
Either there is $\pi \in S_k$ such that $b=\pi^*(b')$; or else for at least a proportion $1-O(kp^{-1})$ of the $u\in \Z_{p-1}$ we have
\begin{equation}\label{eqn.weilaverage}
\left| \sum_{t\in\Z_p} \prod_{i=1}^k \chi^{-a_i-u}(b_i+t) \overline{\prod_{i=1}^{k}\chi^{-a_i'-u}(b_i'+t)}\right|=O(k\sqrt{p}).
\end{equation}
\end{claim*}
\begin{proof}
To see this we will apply the Weil bound (Theorem \ref{thm.weilbound}) using that $\chi$ has order $p-1$. First, let $\xi_1,\dots,\xi_d$ be the distinct values of $-b_1,\dots,-b_k,-b_1',\dots,-b_k'$. The summand in (\ref{eqn.weilaverage}) then has the form $\chi(R(t))$ where $R(t)$ is a product of $d \leq 2k$ factors $(t-\xi_i)^{e_i}$ and
\begin{equation*}
e_i(u):=\sum_{l \in [k]:-b_l'=\xi_i}{(a_l'+u)}-\sum_{l \in [k]:-b_l=\xi_i}{(a_l+u)} .
\end{equation*}
Suppose that there is no $\pi \in S_k$ such that $b=\pi^*(b')$. Then there is some $i \in [d]$ such that $m_i:=|\{l \in [k]: -b_l'=\xi_i\}| - |\{l \in [k]: -b_l=\xi_i\}|$ has $m_i \neq 0$. If $e_i(u) \not \equiv 0 \bmod{p-1}$ then by Theorem \ref{thm.weilbound} we have (\ref{eqn.weilaverage}) for that particular $u$, so set
\[E:=\{u \in \Z_{p-1}: e_i(u) \equiv 0 \bmod{p-1}\}.\]
If $|E|>2k$ then there are $u,u' \in E$ with $u \neq u'$ and $\left\|\frac{u-u'}{p-1}\right\|_\T<\frac 1{2k}$. Since $|m_i| \leq k$, we have
\begin{equation*}
0<\left\|m_i\frac{u-u'}{p-1}\right\|_\T=\left\|\frac{e_i(u)-e_i(u')}{p-1}\right\|_\T=0
\end{equation*}
which is a contradiction. Hence $|E| \leq 2k$, and the result is proved.\end{proof}

Suppose that $(j,j') \in \mathcal{S}^c$. Writing $h=pu-(p-1)t$ for $h\in \Z_N$, $u \in \Z_{p-1}$ and $t \in \Z_p$ we have
\begin{align*}
\sum_{ h \in \Z_N}{g(j+h\textbf1)\overline{g(j'+h\textbf 1)}} & = p^{-k}\sum_{u \in \Z_{p-1}}{\prod_{i=1}^k{G(\psi_1,\chi^{a_i+u})\overline{G(\psi_1,\chi^{a_i'+u})}}}\\ & \qquad \qquad \qquad \times  \sum_{t \in \Z_p}{\prod_{i=1}^k\chi^{-a_i-u}(b_i+t)\overline{\chi^{-a_i'-u}(b_i'+t)}}.
\end{align*}
Inserting the bound $|G(\psi_1,\chi^{a_i+u})| \leq \sqrt{p}$ along with the claim (where we are in the second case since $(j,j') \in \mathcal{S}^c$) gives
\begin{equation*}
\sum_{ h \in \Z_N}{g(j+h\textbf1)\overline{g(j'+h\textbf 1)}}  \leq p^{-k}\cdot (p-1)\cdot \sqrt{p}^{2k}\cdot O_k(\sqrt{p}) + p^{-k}\cdot O_k(1) \cdot \sqrt{p}^{2k}\cdot p =O_k(p^{\frac{3}{2}}).
\end{equation*}
Using this in (\ref{eqn.fA}), the fact that $\mathcal{S}^c$ is closed under $(j,j')\mapsto (j+h\textbf 1,j'+h\textbf 1)$ and $H(j,j')$ is invariant under $(j,j')\mapsto (j+h\textbf 1,j'+h\textbf 1)$, we get
\begin{align}
\nonumber \int_0^1|f_A(\theta)|^{2k}\dd\theta
&=
\frac1{N^{2k}}\left(\sum_{(j,j') \in \mathcal{S}}{g(j)\overline{g(j')}H(j,j')}\right.\\ \nonumber & \qquad\qquad \qquad \left. + \sum_{(j,j') \in \mathcal{S}^c}{H(j,j')\frac{1}{N}\sum_{h \in \Z_N}{g(j+h \textbf1)\overline{g(j'+h\textbf 1)}}}\right)\\ \nonumber
&=
\frac1{N^{2k}}\left(\sum_{(j,j') \in \mathcal{S}}{g(j)\overline{g(j')}H(j,j')} + \sum_{j,j' \in \Z_N^k}{\frac{1}{N}O_k(p^{\frac32})|H(j,j')|}\right)\\ & = \frac1{N^{2k}}\sum_{(j,j') \in \mathcal{S}}{g(j)\overline{g(j')}H(j,j')} + O_k(p^{-\frac12}\log^{2k}N).
\end{align}
The passage to the last line uses $N=p(p-1)$ and (\ref{eqn.dirichletsum2}) $2k$ times, once for each of the variables $j_1$, $\dots$, $j_k$, $j_1'$, $\dots$, $j_k'$. (This is similar to the proof of (\ref{Hjj}).) Finally, if $j,j' \in \Z_N^k$ and $\pi \in S_k$ is such that $b=\pi^*(b')$ but $j\neq \pi^*(j')$, then $b=\pi^*(b')$ implies $j_i \equiv j_{\pi(i)}' \bmod p$ for all $i \in [k]$, and hence there is some $i \in [k]$ such that $j_i \neq j_{\pi(i)}'$ and $p \mid j_i - j_{\pi(i)}'$. By (\ref{Hjj2}) (and the fact that $|g(j)|,|g(j')| \leq 1$) we conclude that
\begin{align*}
\left|\sum_{\substack{j,j' \in \Z_N^k\\\exists \pi \in S_k \text{ s.t. }b=\pi^*(b')\\ \text{ and }j \neq \pi^*(j')}}{g(j)\overline{g(j')}H(j,j')}\right| & \leq \sum_{\pi \in S_k}{\sum_{i \in [k]}{\sum_{\substack{j,j' \in \Z_N^k \\ j_i \neq j'_{\pi(i)}\\ p \mid j_i - j'_{\pi(i)}}}{|g(j)||g(j')||H(j,j')|}}}\\
 &= O_k(p^{-1}N^{2k}\log^{2k}N).
\end{align*}
Hence
\begin{equation*}
\int_0^1|f_A(\theta)|^{2k}\dd\theta=\frac{1}{N^{2k}}\sum_{\substack{j,j' \in \Z_N^k\\\exists \pi \in S_k \text{ s.t. }j= \pi^*(j')}}{g(j)\overline{g(j')}H(j,j')}+ O_k(p^{-\frac12}\log^{2k}N).
\end{equation*}
Finally, for $j,j' \in \Z_N^k$ with $j=\pi^*(j')$ we have $g(j)\overline{g(j')}=1$ if $p \nmid j_r$ and $p-1 \nmid j_r$ for all $r \in [k]$, and otherwise $|g(j)\overline{g(j')}| \leq 1$. The result then follows from (\ref{Hjj}) and (\ref{eqn.Fmomentexpansion}).
\end{proof}

\begin{proof}[Proof of Proposition \ref{prop.moments}]\label{proof.moments}
Let $Z_p(\theta):=|f_A(\theta)|^2$ where $A$ is as in (\ref{eqn.ruzsa}), and $Y_p(\omega,\theta):=|F(\theta)|^2$ where $\omega \in \Omega$ and $\Omega$ is a sample space for the random variables $X_1,\dots,X_{N-1}$ where $N=p(p-1)$. We think of $Z_p$ as a random variable on $[0,1]$ and $Y_p$ as a random variable on $\Omega \times [0,1]$. From Proposition \ref{prop.moments2k} we have
\begin{equation}\label{eqn.momentcond}
\E{Z_p^{k}} =\E{Y_p^{k}}+o_{k; p \rightarrow \infty}(1).
\end{equation}
By (\ref{eqn.stein2k}) from Proposition \ref{prop.stein2k} we have $\E{Y_p^k} \leq k!$ and hence by the previous $\E{Z_p^k}\leq k!+o_{k;p \rightarrow \infty}(1)$.

This fact for $k=1$ implies that the laws corresponding to $(Z_p)_p$ and $(Y_p)_p$ are tight (as in the definition \cite[p336]{bil::0}) since $\P(|Z_p|\geq R) \leq R^{-1}\E{Z_p} = O(R^{-1})$ and similarly for the $Y_p$s. Hence for an arbitrary sequence of primes $(p_i)_i$, by Prokhorov's theorem \cite[Theorem 25.10, p.~336]{bil::0}, if $(\mu_i)_i$ (resp.\ $(\nu_i)_i$) has $\mu_i$ (resp.\ $\nu_i$) as the law of $Z_{p_i}$ (resp.\ $Y_{p_i}$), then they both have weakly convergent subsequences $\mu_{i_j} \rightarrow \mu$ and $\nu_{i_j} \rightarrow \nu$. (To ensure it is the same subsequence first pass to a convergent subsequence of the $\mu_i$s, and then pass to a subsequence of that subsequence along which the $\nu_i$s converge.)
 
 By (\ref{eqn.momentcond}) (with $k$ replaced by $2k$) we have that $(Z_{p}^{k})_p$ is uniformly integrable \cite[(25.10), p.~338]{bil::0} since
 \begin{equation*}
 \E{Z_p^k1_{[Z_p^k \geq R]}} \leq R^{-1}((2k)!+o_{k; p \rightarrow \infty}(1)) = O_k(R^{-1}),
 \end{equation*}
 and similarly for $(Y_p^{k})_p$. Since $\mu_{i_j}$ and $\nu_{i_j}$ converge weakly, we then have from \cite[Theorem 25.12]{bil::0} that
 \begin{equation*}
 \int_0^\infty{x^{k}\dd \mu(x)} = \int_0^\infty{x^{k}\dd\nu(x)} \leq k!.
 \end{equation*}
Hence, by Carleman's condition there is a unique solution to the moment problem and $\mu=\nu$. (For a specific reference we could use \cite[Theorem 30.1, p.~388]{bil::0}.) Since $(\sqrt{Z_{p}})_p$ and $(\sqrt{Y_p})_p$ are uniformly integrable (from the uniformly bounded first moments) and $\mu_{i_j}$ and $\nu_{i_j}$ converge weakly we have from \cite[Theorem 25.12]{bil::0} that
 \begin{equation*}
 \lim_{j \rightarrow \infty}{\E{\sqrt{Z_{p_{i_j}}}}-\E{\sqrt{Y_{p_{i_j}}}}}=0.
 \end{equation*}
Since every subsequence has a further subsequence converging to $0$, the sequence itself tends to $0$. This gives the result.
\end{proof}

\subsection{Moments of the random model}\label{sec.stein}

In this section we prove two results:
\begin{proposition*}[Proposition \ref{prop.stein2k}]
For $k\in \N$ we have
\begin{equation}\label{eqn.stein2k}
\int_0^1\E|F(\theta)|^{2k}\dd\theta \leq k!.
\end{equation}
\end{proposition*}
\begin{proof}
By interchanging the order of summation we have
\begin{equation*}
\sum_{j=1}^{N-1}{\left|\frac{1}{N}\wh{1_{[N]}}\left(\theta - \frac{j}{N}\right)\right|^2} \leq 1.
\end{equation*}
From the definition of $F$ in (\ref{eqn.Fmod}) and $\E{X_j^r\overline{X_j}^s}=1_{[r=s]}$ we have
\begin{equation*}
\E{|F(\theta)|^{2k}}= \sum_{a_1+\cdots +a_{N-1}=k}{\binom{k}{a_1,\dots,a_{N-1}}^2\prod_{j=1}^{N-1}{\left|\frac{1}{N}\wh{1_{[N]}}\left(\theta - \frac{j}{N}\right)\right|^{2a_j}}}.
\end{equation*}
Since the multinomial coefficient is at most $k!$ we then have
\begin{align*}
\E{|F(\theta)|^{2k}}& \leq  k!\sum_{a_1+\cdots +a_{N-1}=k}{\binom{k}{a_1,\dots,a_{N-1}}\prod_{j=1}^{N-1}{\left|\frac{1}{N}\wh{1_{[N]}}\left(\theta - \frac{j}{N}\right)\right|^{2a_j}}} \\ & = k!\left(\sum_{j=1}^{N-1}{\left|\frac{1}{N}\wh{1_{[N]}}\left(\theta - \frac{j}{N}\right)\right|^2}\right)^k \leq k!.
\end{align*}
The proposition is proved.
\end{proof}

\begin{proposition*}[Proposition \ref{prop.stein}]
\begin{equation*}
\int_0^1{\E{|F(\theta)|}\dd \theta} \rightarrow c_\mathcal{R}:=\int_0^1{\E{\left|\sum_{m\in \Z}{Y_m\sinc (\pi(s-m))}\right|}\dd s} \text{ as }N \rightarrow \infty
\end{equation*}
where $(Y_m)_{m \in \Z}$ are independent Steinhaus random variables.
\end{proposition*}
\begin{proof}
For $s \in [0,1)$, $m \mapsto \sinc(\pi(s-m))$ is the Fourier transform of $\theta \mapsto \exp(2\pi i s(\theta-\frac12))$ on the interval $[0,1]$ and so by Parseval's theorem we get
\begin{equation}\label{opening}
\sum_{m\in \Z}{\sinc^2 (\pi(s-m))} = 1.
\end{equation}
Write $\Omega$ for the common sample space for the $Y_m$s.  The random variables $Y_m\sinc (\pi(s-m))$ are independent mean zero random variables with summable variance and so
\begin{equation}\label{eqn.asc}
S_M(\omega,s):=\sum_{|m| \leq M}{Y_m\sinc (\pi(s-m))} \rightarrow S(\omega,s):=\sum_{m \in \Z}{Y_m\sinc (\pi(s-m))}
\end{equation}
for every $s \in [0,1)$ and almost every $\omega \in \Omega$.  It follows by Fubini's theorem that $S_M \rightarrow S$ a.e.\ $(\omega,s) \in \Omega \times [0,1)$.

Write $t=N\theta$ to see that
\begin{equation}\label{eqn.firstcalc}
\int_0^1{\E{|F(\theta)|}\dd\theta}= \frac{1}{N}\int_{0}^N{ \E{\left| \sum_{j=1}^{N-1}{X_j\frac{\wh{1_{[N]}}((t-j)/N)}{N}}\right|}\dd t}.
\end{equation}
Now, for any $\theta \in (-1,1)$ we have
\begin{equation}\label{eqn.dk}
\wh{1_{[N]}}(\theta)=\exp(-\pi i (N+1)\theta)\frac{N\sinc (\pi N \theta)}{\sinc (\pi \theta)}.
\end{equation}
Since $X_1$, \dots, $X_{N-1}$ are independent uniform $S^1$-valued random variables, $\omega_1X_1$, \dots, $\omega_{N-1}X_{N-1}$ are independent uniform $S^1$-valued random variables for any $\omega_1,\dots,\omega_{N-1} \in S^1$. Hence from (\ref{eqn.dk}), putting $\omega_j:=\exp(-\pi i \frac{N+1}{N}(t-j))$, we have
\begin{equation*}
 \E{\left| \sum_{j=1}^{N-1}{X_j\frac{\wh{1_{[N]}}((t-j)/N)}{N}}\right|}= \E{\left| \sum_{j=1}^{N-1}{X_j \frac{\sinc(\pi (t-j))}{\sinc(\pi(t-j)/N)}}\right|}.
\end{equation*}
Inserting this in (\ref{eqn.firstcalc}) and writing $t=k+s$ for $s \in [0,1)$ and $k \in \Z$ we get
\begin{align}
\int_0^1{\E{|F(\theta)|}\dd\theta}
&= \frac{1}{N}\sum_{k=0}^{N-1}{\int_0^1{ \E{\left| \sum_{j=1}^{N-1}{X_j \frac{\sinc(\pi(k+s-j))}{\sinc(\pi(k+s-j)/N)}}\right|}\dd s}} \label{eqn.calc}.
\end{align}
To estimate this it is useful to have an auxiliary estimate: for $S \subset [N-1]$ we have
\begin{align}\nonumber 
\left(\E{\left| \sum_{ j \in S}{X_j \frac{\sinc(\pi (k+s-j))}{\sinc(\pi(k+s-j)/N)}}\right|}\right)^2 & \leq \E{\left| \sum_{ j \in S}{X_j \frac{\sinc(\pi (k+s-j))}{\sinc(\pi(k+s-j)/N)}}\right|^2}\\ \nonumber &= \sum_{j \in S}{\frac{\sinc^2( \pi (k+s-j))}{\sinc^2 (\pi (k+s-j)/N)}}\\& \leq \frac{1}{N^2}\sum_{ j \in S}{\frac{1}{\sin^2(\pi (k+s-j)/N)}}\nonumber \\
& =O\left(\sum_{ j \in S}{\frac{1}{1+\min\{|k-j|,(N-|k-j|)\}^2}}\right). \label{eqn.l2dump}
\end{align}
Fix $M$, to be sent to infinity later. We delete those terms on the right in (\ref{eqn.calc}) with $k \leq M$ and $k \geq N-M$. By (\ref{eqn.l2dump}) with $S=[N-1]$ these contribute $O(M/N)$ to the right of (\ref{eqn.calc}). For each $M<k<N-M$ we then use the triangle inequality to delete those $j$s in the inner sum with $|k-j|>M$. By (\ref{eqn.l2dump}) with $S:=\{j \in [N-1]: |k-j|>M\}$ these contribute $O(M^{-1/2})$ to the right of (\ref{eqn.calc}). Combining all this we get
\begin{align*}
\int_0^1{\E{|F(\theta)|}\dd\theta} & = \frac{1}{N}\sum_{M<k<N-M}{\int_0^1{ \E{\left| \sum_{m=-M}^{M}{X_{k+m}\frac{\sinc(\pi(s-m))}{\sinc(\pi(s-m)/N)}}\right|}\dd s}}\\ &\qquad \qquad \qquad \qquad \qquad \qquad \qquad \qquad \qquad + O\left(MN^{-1}+M^{-1/2}\right).
\end{align*}
Now $X_{k-M},\dots,X_{k+M}$ are independent uniform $S^1$-valued random variables. Hence if $(Y_m)_{m \in \Z}$ are independent uniform $S^1$-valued random variables the inner expectation can be written as
\begin{equation*}
 \E{\left| \sum_{m=-M}^{M}{X_{k+m}\frac{\sinc(\pi(s-m))}{\sinc(\pi(s-m)/N)}}\right|}= \E{\left| \sum_{m=-M}^{M}{Y_m\frac{\sinc(\pi (s-m))}{\sinc(\pi(s-m)/N)}}\right|}.
\end{equation*}
In particular it is independent of $k$ and so we have
\begin{align}\label{eqn.mainest}
\int_0^1{\E{|F(\theta)|}\dd\theta} &= \left(1-\frac{2M+1}{N}\right)\int_0^1{ \E{\left| \sum_{m=-M}^{M}{Y_m\frac{\sinc(\pi (s-m))}{\sinc(\pi(s-m)/N)}}\right|}\dd s} \\ &\qquad \qquad \qquad \qquad \qquad \qquad \qquad \qquad \qquad + O(MN^{-1}+M^{-1/2}). \nonumber
\end{align}
For fixed $M$, $-M \leq m \leq M$, and $s \in [0,1)$ we have $\frac{\sinc(\pi (s-m))}{\sinc (\pi (s-m)/N)} \rightarrow \sinc(\pi (s-m))$ as $N \rightarrow \infty$. Hence
\begin{equation*}
\sum_{m=-M}^{M}{Y_m(\omega)\frac{\sinc(\pi (s-m))}{\sinc(\pi(s-m)/N)}} \rightarrow S_M(\omega,s)
\end{equation*}
for almost every $\omega \in \Omega$. As before that means by Fubini's theorem that the left hand side converges to $S_M(\omega,s)$ for a.e.\ $(\omega,s) \in \Omega \times [0,1)$. Once $N>2M$ the left hand side is also bounded independently of $N$, $\omega$, and $s$, and so by bounded convergence we have
\begin{equation*}
\int_0^1{ \E{\left| \sum_{m=-M}^{M}{Y_m\frac{\sinc(\pi (s-m))}{\sinc(\pi(s-m)/N)}}\right|}\dd s} \rightarrow \int_0^1{ \E{\left| S_M(\omega,s)\right|}\dd s}
\end{equation*}
as $N \rightarrow \infty$. Hence
\begin{equation*}
\limsup_{N \rightarrow \infty}{\left| \int_0^1{\E{|F(\theta)|}\dd\theta} -\int_0^1{ \E{\left| S_M(\omega,s)\right|}\dd s}\right|} =O(M^{-1/2}).
\end{equation*}
Finally, from (\ref{opening}),
\begin{equation*}
 \E{\left| S_M(\omega,s)\right|^2} =\sum_{|m| \leq M}{\sinc^2(\pi (s-m))} \leq \sum_{m \in \Z}{\sinc^2(\pi (s-m))}=1.
\end{equation*}
It follows that $\|S_M\|_{L_2(\Omega\times [0,1))} \leq 1$ and hence the random variables $S_M$ are uniformly integrable. In view of the almost sure convergence in (\ref{eqn.asc}), the random variables $S_M$ converge in $L_1(\Omega\times [0,1))$ as required. The result is proved.
\end{proof}

\section{Limitations of the approach}\label{sec.limitations}

The approach of this paper to finding sets of integers $A$ such that $\|\wh{1_{A}}\|_1$ is large has been to first find a subset $\mathcal{A}$ of a cyclic group $\Z_N$ for which $\|\wh{1_{\mathcal{A}}}\|_1$ is large and then lift it to the integers $[N]$.

The only examples we know of subsets of cyclic groups for which $\|\wh{1_\mathcal{A}}\|_1$ is close to the maximum described in Proposition \ref{prop.upgen}, also satisfy
\begin{equation}\label{eqn.robperf}
\|1_\mathcal{A} \ast 1_{-\mathcal{A}} - 1_{\Z_N}\|_{\ell_1(\Z_N)} \leq Kn \text{ where }n=|\mathcal{A}|
\end{equation}
for some small $K$. Since $1_\mathcal{A} \ast 1_{-\mathcal{A}}(0)=n$ we always have $K \geq 1-\frac{1}{n}$, and if $\mathcal{A}$ is a perfect difference set we have equality.

In general we think of $K=o_{n \rightarrow \infty}(n)$, so that $N\sim n^2$. This is similar to the condition of being a dense Sidon set in the sense of \cite{Eberhard:2023aa}, and we refer the reader there for numerous examples.

It turns out that lifts of such sets cannot show that $c_\textsc{b}=0$ as we shall see in Proposition \ref{prop.limit} below. This also explains the role of (\ref{eqn.robperf}): it both captures all the examples we know, and is enough to ensure that $\wh{1_\mathcal{A}}$ is small away from the identity which is what makes the proof below work. It is not, as far as we know, natural in any more significant way.
\begin{proposition}\label{prop.limit}
Suppose that $\mathcal{A} \subset \Z_N$ has size $n>1$ and satisfies (\ref{eqn.robperf}) for some $K$. Then writing $A$ for the lift of $\mathcal{A}$ in $[N]$, we have
\begin{equation*}
\|\wh{1_A}\|_1 \leq \left(1-\frac{\newconstlittle{mz}}{K^2}\right)\sqrt{n}.
\end{equation*}
\end{proposition}
\begin{proof}
Since $n>1$, we have $K \geq 1-\frac{1}{n}\geq \frac{1}{2}$ and also (since $n>1$) Proposition \ref{prop.upgen} tells us that
\begin{equation*}
\|\wh{1_A}\|_1 \leq \sqrt{n}\left(1-\frac{\newconstlittle{up}}{n}\right).
\end{equation*}
This means we can make some asymptotic assumptions about our parameters at the cost of reducing $\refconstlittle{mz}$. Specifically we may assume that
\begin{equation}\label{eqn.nlargee}
K \leq \frac{1}{216\refconstbig{claim}^3}\sqrt{n} \text{ and } \frac{\log n}{\sqrt{n}} \leq \frac{1}{1152(C_6+4)\refconstbig{k}}
\end{equation}
by taking (since $\log n\leq 4n^{\frac{1}{4}}$)
\begin{equation*}
\refconstlittle{mz}=\min{\left\{\frac{\refconstlittle{up}}{216^2\refconstbig{claim}^6},
\frac{\refconstlittle{up}}{1024(1152(C_6+4)\refconstbig{k})^4}\right\}}.
\end{equation*}
We pair $\Z_N$ with itself via the map $(r,x) \mapsto \exp(2\pi i rx/N)$. Then
\begin{equation}\label{eqn.sub}
\wh{1_{\Z_N}}(r)=\begin{cases} N &\text{ if }N \mid r\\ 0 & \text{ otherwise.}\end{cases}
\end{equation}
Hence by (\ref{eqn.robperf}) and the Hausdorff-Young inequality
\begin{equation*}
|\wh{1_\mathcal{A}}(r)|^2 = |(1_{\mathcal{A}} \ast 1_{-\mathcal{A}} -1_{\Z_N})^\wedge(r)| \leq Kn \text{ whenever }N \nmid r .
\end{equation*}
This rearranges to
\begin{equation}
\label{eqn.upt}
|\wh{1_\mathcal{A}}(r)| \leq \sqrt{Kn}\text{ whenever }N \nmid r.
\end{equation}
We also have
\begin{equation}
\label{eqn.other}
\frac{1}{\sqrt{n}}\wh{1_\mathcal{A}}(0) - \frac{\sqrt{n}}{N}\wh{1_{\Z_N}}(0)=0.
\end{equation}
Since $A$ is the lift of $\mathcal{A}$ to $[N]$, and $[N]$ is the lift of $\Z_N$ to $[N]$, we have
\begin{equation*}
\wh{1_A}(r/N) = \wh{1_{\mathcal{A}}}(r) \text{ and } \wh{1_{[N]}}(r/N)=\wh{1_{\Z_N}}(r) \text{ for all }r.
\end{equation*}
Writing
\begin{equation*}
f(\theta):=\frac{1}{\sqrt{n}}\wh{1_A}(\theta) - \frac{\sqrt{n}}{N}\wh{1_{[N]}}(\theta) \text{ for all }\theta \in \T,
\end{equation*}
we have from (\ref{eqn.upt}) and (\ref{eqn.other}) that
\begin{align}
 |f(r/N)|&=\left|\frac{1}{\sqrt{n}}\wh{1_\mathcal{A}}(r) - \frac{\sqrt{n}}{N}\wh{1_{\Z_N}}(r)\right| \leq \sqrt{K} \text{ for all }r;\label{eqn.inf}
\end{align}
and by Parseval's theorem for $\Z_N$ we have that
\begin{equation}\label{eqn.paro}
\frac{1}{N}\sum_{r=1}^N{|f(r/N)|^2} = \left\|\frac{1}{\sqrt{n}}1_\mathcal{A} - \frac{\sqrt{n}}{N}1_{\Z_N}\right\|_{\ell_2(\Z_N)}^2 = 1-\frac{n}{N} \leq 1.
\end{equation}
Note that Parseval's theorem for $\Z$ also gives
\begin{equation}\label{eqn.paro2}
\int_0^1{|f(\theta)|^2\dd\theta}=\left\|\frac{1}{\sqrt{n}}1_A - \frac{\sqrt{n}}{N}1_{[N]}\right\|_{\ell_2(\Z)}^2=1-\frac{n}{N}
\end{equation}
since $A \subset [N]$.

Define the polynomial
\begin{equation*}
P(z)=\frac{1}{\sqrt{n}}\sum_{a \in A}{z^{a-1}} - \frac{\sqrt{n}}{N}\sum_{j=1}^N{z^{j-1}} 
\end{equation*}
which is a polynomial (since $a \geq 1$ for all $a \in A$) of degree at most $N-1$. Moreover,
\begin{equation*}
f(\theta)=\exp(2\pi i \theta)P(\exp(2\pi i \theta))
\end{equation*}
and so
\begin{equation*}
|f(\theta)|=|P(\exp(2\pi i \theta))|.
\end{equation*}
Now apply the Marcinkiewicz–Zygmund inequality for interpolating polynomials (not for sums of random variables). Specifically, apply Theorem \ref{thm.zyg}, case (\ref{thm.pt2}), to the polynomial $P(z)$ with $p=6$ to get
\begin{equation*}
\int_0^1{|f(\theta)|^6\dd\theta}=\int_0^1{|P(\exp(2\pi i \theta))|^6\dd\theta} \leq C_6\frac{1}{N}\sum_{r=1}^N{|f(r/N)|^6}\leq C_6K^2.
\end{equation*}
The last inequality here is by (\ref{eqn.inf}) and (\ref{eqn.paro}).
\begin{claim} 
\begin{equation*}
\int_0^1{|f(\theta)|^4\dd\theta} \geq \frac{4}{3}-\newconstbig{claim}K^{\frac{1}{3}}n^{-\frac{1}{6}}.
\end{equation*}
\end{claim}
\begin{proof}
We use the inequality
\begin{equation*}
||a-b|^4-|a|^4+|b|^4|\leq 4|a-b|^3|b|+6|a-b|^2|b|^2+4|a-b||b|^3 \text{ for all }a,b \in \C.
\end{equation*}
Put $a=\frac{1}{\sqrt{n}}\wh{1_A}(\theta)$ and $b=\frac{\sqrt{n}}{N}\wh{1_{[N]}}(\theta)$ and integrate over $\theta$. Then
\begin{align}
\nonumber&\left|\int_0^1{|f(\theta)|^4\dd\theta}- \frac{1}{n^2}\int_0^1{|\wh{1_A}(\theta)|^4\dd \theta} + \frac{n^2}{N^4}\int_0^1{|\wh{1_{[N]}}(\theta)|^4\dd \theta}\right| \\ 
\nonumber&\qquad \qquad \qquad \qquad \leq \int_0^1{\left||f(\theta)|^4- \frac{1}{n^2}|\wh{1_A}(\theta)|^4 + \frac{n^2}{N^4}|\wh{1_{[N]}}(\theta)|^4\right|\dd\theta} \\ \nonumber &\qquad \qquad  \qquad \qquad\leq \int_0^1{\frac{4n^{1/2}}{N}|f(\theta)|^3|\wh{1_{[N]}}(\theta)| + 6  \frac{n}{N^2}|f(\theta)|^2|\wh{1_{[N]}}(\theta)|^2}\\ &\label{eqn.u5y} \qquad \qquad\ \qquad \qquad\qquad \qquad\qquad \qquad\qquad \qquad + 4\frac{n^{3/2}}{N^3}|f(\theta)||\wh{1_{[N]}}(\theta)|^3\dd\theta.
\end{align}
Now
\begin{equation*}
\|\wh{1_{[N]}}\|_p^p =\int_0^1{|\wh{1_{[N]}}(\theta)|^p\dd\theta}\leq \|\wh{1_{[N]}}\|_\infty^{p-2}\|\wh{1_{[N]}}\|_2^2 \leq N^{p-1} \text{ whenever }p\geq 2.
\end{equation*}
Hence $\|\wh{1_{[N]}}\|_p \leq N^{1-\frac{1}{p}}$ for $p\geq 2$. H{\"o}lder's inequality then gives
\begin{equation*}
\int_0^1{|f(\theta)|^3|\wh{1_{[N]}}(\theta)|\dd\theta} \leq \|f\|_6^3\|\wh{1_{[N]}}\|_2=O(KN^{\frac{1}{2}});
\end{equation*}
\begin{equation*}
 \int_0^1{|f(\theta)|^2|\wh{1_{[N]}}(\theta)|^2\dd\theta} \leq \|f\|_6^2\|\wh{1_{[N]}}\|_3^{2}=O(K^{\frac{2}{3}}N^{\frac{4}{3}});
\end{equation*}
and also
\begin{equation*}
\int_0^1{|f(\theta)||\wh{1_{[N]}}(\theta)|^3\dd\theta} \leq \|f\|_6 \|\wh{1_{[N]}}\|_{\frac{18}{5}}^3  = O(K^{\frac{1}{3}}N^{\frac{13}{6}}).
\end{equation*}
The three previous estimates in (\ref{eqn.u5y}) along with $N^{-1}=O(n^{-2})$ and $K=O(\sqrt{n})$ give
\begin{equation*}
\left|\int_0^1{|f(\theta)|^4\dd\theta}-\frac{1}{n^2}\int_0^1{|\wh{1_A}(\theta)|^4\dd\theta} + \frac{n^2}{N^4}\int_0^1{|\wh{1_{[N]}}(\theta)|^4\dd\theta}\right| = O(K^{\frac{1}{3}}n^{-\frac{1}{6}}).
\end{equation*}
Finally, $\int_0^1{|\wh{1_A}(\theta)|^4\dd\theta}$ is the additive energy of $A$ which is at least $2n^2-n$. Then following the calculation at the top of \cite[p.~945]{Shao:2026aa} for the additive energy of an interval gives:
\begin{align}\nonumber
\int_0^1{|\wh{1_{[N]}}(\theta)|^4\dd\theta}& = \int_0^1{\wh{1_{[N]}}(\theta)^2\overline{\wh{1_{[N]}}(\theta)^2}\dd\theta}\\ \nonumber & =\sum_{z \in \Z}{1_{[N]} \ast 1_{[N]}(z)^2}\\ \nonumber & =\sum_{z\in \Z}{|\{(s,t): s+t=z, 1 \leq s,t \leq N\}|^2}\\ \label{eqn.quadcount} & = 1^2+\dots + N^2+(N-1)^2+\dots+1^2 = \frac{2N^3+N}{3}.
\end{align}
Inserting these gives the claimed result, using again that $N^{-1}=n^{-2}+O(Kn^{-3})$.
\end{proof}
Now we follow the derivation in \cite[(3.3), p.~61]{bou::7}. First,
\begin{align*}
\int_0^1{(|f(\theta)|^2-1)^2\dd\theta} & = \int_0^1{(|f(\theta)|-1)\cdot ((|f(\theta)|^2-1)(|f(\theta)|+1))\dd\theta}\\ & \leq \left(\int_0^1{(|f(\theta)|-1)^2\dd\theta}\right)^{\frac{1}{2}} \left(\int_0^1{((|f(\theta)|^2-1)(|f(\theta)|+1))^2\dd\theta}\right)^{\frac{1}{2}}\\ & \leq  \left(\int_0^1{(|f(\theta)|-1)^2\dd\theta}\right)^{\frac{1}{2}} \left(8\int_0^1{(|f(\theta)|^6+1)\dd\theta}\right)^{\frac{1}{2}}.
\end{align*}
From the claim and the fact that $K\leq \frac{1}{216 \refconstbig{claim}^3}\sqrt{n}$ and (\ref{eqn.paro2}) we have that
\begin{equation*}
\frac{1}{6} \leq \|f\|_4^4-1\leq \|f\|_4^4+1-2(1-\frac{n}{N}) = \int_0^1{(|f(\theta)|^2-1)^2\dd\theta}.
\end{equation*}
From (\ref{eqn.paro2}) again we also have
\begin{equation*}
\int_0^1{(|f(\theta)|-1)^2\dd\theta} = 1+\int_0^1{|f(\theta)|^2\dd\theta}-2\|f\|_1 \leq 2(1-\|f\|_1)).
\end{equation*}
Finally since $K \geq \frac{1}{2}$ we have
\begin{equation*}
\int_0^1{(|f(\theta)|^6+1)\dd\theta} \leq (C_6+4)K^2.
\end{equation*}
Combining all these gives
\begin{equation*}
 \|f\|_1 \leq 1-\frac{1}{576(C_6+4)K^2}.
\end{equation*}
By \cite[(12.1), Chapter II, Vol.~1, p.~67]{zyg::} and the fact $N^{-1}=O(n^{-2})$ and $n^{-1}=O(K^{-2})$ we have
\begin{equation*}
\frac{n}{N}\|\wh{1_{[N]}}\|_1 \leq \sqrt{n}\cdot K^{-2}\cdot \newconstbig{k}\frac{\log n}{\sqrt{n}}.
\end{equation*}
Finally,
\begin{align*}
\|\wh{1_A}\|_1& \leq \sqrt{n}\|f\|_1 + \frac{n}{N}\|\wh{1_{[N]}}\|_1\\ & \leq \sqrt{n}\left(1-\frac{1}{K^2}\left(\frac{1}{576(C_6+4)} -\refconstbig{k}\frac{\log n}{\sqrt{n}}\right)\right).
\end{align*}
The result is proved given the assumed size of $n$ in (\ref{eqn.nlargee}).
\end{proof}

For the record \cite[Theorem 7.10, Chapter XI, Vol 2, p.~30]{zyg::} (see \cite[(1.18), Chapter X, Vol.~2, p.~4]{zyg::} for the definition of $\omega_{n+1}$ there) states the following:
\begin{theorem}\label{thm.zyg}
There is $C>1$ and for each $p \in (0,\infty)$ there is $C_p>1$ such that if $N \in \N^*$, and $P$ is a degree $N-1$ polynomial then
\begin{equation}\label{thm.pt1}
\left(\frac{1}{N}\sum_{r=0}^{N-1}{|P(\exp(2\pi i r/N))|^p}\right)^{1/p} \leq C\left(\int_0^1{|P(\exp(2\pi i \theta))|^p\dd\theta}\right)^{1/p} \text{ for all }1 \leq p \leq \infty;
\end{equation}
and
\begin{equation}\label{thm.pt2}
\left(\int_0^1{|P(\exp(2\pi i \theta))|^p\dd\theta}\right)^{1/p} \leq C_p^{1/p}\left(\frac{1}{N}\sum_{r=0}^{N-1}{|P(\exp(2\pi i r/N))|^p}\right)^{1/p}  \text{ for all }1 < p < \infty.
\end{equation}
\end{theorem}

\bibliographystyle{halpha}

\bibliography{references}

\end{document}